\documentclass{article}
\usepackage[utf8]{inputenc}
\usepackage{amsmath,amsthm,amsfonts,amssymb,amstext}

\usepackage{caption}
\usepackage{subcaption}
\usepackage{graphicx}

\usepackage{enumerate}

\usepackage{url}

\title{The journey of the union-closed sets conjecture}
\author{Henning Bruhn and Oliver Schaudt}
\date{}

\newtheorem{theorem}{Theorem}
\newtheorem{lemma}[theorem]{Lemma}
\newtheorem{observation}[theorem]{Observation}
\newtheorem{corollary}[theorem]{Corollary}

\newtheorem{conjecture}[theorem]{Conjecture}
\newtheorem{question}[theorem]{Question}

\newcommand{\emtext}[1]{\text{\em #1}}

\newcommand{\fc}{Frankl's conjecture}

\newcommand{\ucc}{union-closed sets conjecture}
\newcommand{\ucf}{union-closed family}
\newcommand{\ucfs}{union-closed families}

\newcommand{\icc}{intersection-closed sets conjecture}
\newcommand{\icf}{intersection-closed family}

\newcommand{\A}{\mathcal A}
\newcommand{\B}{\mathcal B}

\newcommand{\D}{\mathcal D}

\newcommand{\upcl}[1]{[#1)}

\newcommand{\join}{\vee}
\newcommand{\meet}{\wedge}
\newcommand{\sm}{\setminus}

\newcommand{\hord}{<}
\newcommand{\hfam}{\mathcal H^{(n)}}
\newcommand{\pow}[1]{2^{#1}}

\newcommand{\nsub}[1]{_{\nsubseteq #1}}

\newcommand{\selfcite}[1]{$\!\!${\bf\cite{#1}}}
\renewcommand{\epsilon}{\varepsilon}

\usepackage{pgfplots}
\pgfplotsset{width=7cm}

\begin{document}

\maketitle

\begin{abstract}
We survey the state of the union-closed sets conjecture.
\end{abstract}

\section{Introduction}

One of the first mentions~\cite{AMG87a} of the union-closed sets conjecture
calls it  ``a much-travelled conjecture''. This is indeed so. 
Geographically it has 
spread from Europe to at least North America, Asia, Oceania and Australia. 
Mathematically it has ventured from its origins in extremal set theory to 
lattice and graph theory.
In this survey we strive to trace its journey.

The main  attraction of the conjecture is certainly its 
simple formulation.
A
family $\A$ of sets is \emph{union-closed} if for every two
member-sets $A,B\in\A$ also their union $A\cup B$ is contained in $\A$.

\newtheorem*{ucsc}{Union-closed sets conjecture}
\begin{ucsc}
Any finite union-closed family of sets $\A\neq\{\emptyset\}$ 
has an element
that is contained in at least half of the member-sets.
\end{ucsc}

An example of a union-closed family is given in Figure~\ref{fig:sub1}, 
where we have omitted commas and parentheses. There, one may count that 
the elements $1,2,3$ appear each in only $12$ of the $25$ member-sets, 
which is less than half of the sets. 
Each of the other elements $4,5,6$ however 
is contained in  $16$ sets, more than enough for the family to satisfy 
the conjecture. Power sets are other examples of union-closed families,
and there the conjecture is tight: every element appears in exactly 
half of the member-sets.

\begin{figure}
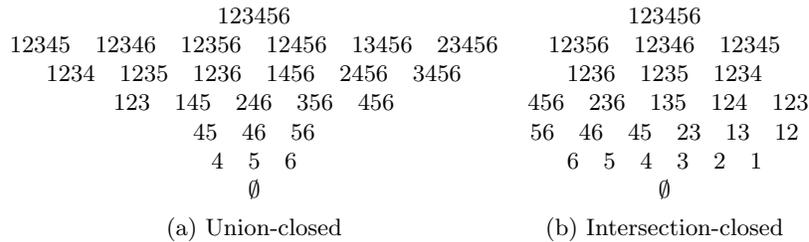

\centering
\begin{subfigure}{.6\textwidth}
  \centering\small
$123456$\\
$12345\quad 
12346\quad 12356\quad 12456\quad 
13456\quad 23456$\\
$1234\quad 1235\quad 1236\quad 
1456\quad 2456\quad 3456$\\ 
$123\quad 145\quad 
246\quad 356\quad 456$\\
$45\quad 46\quad 56$\\
$4\quad 5\quad 6$\\
$\emptyset$
  \caption{Union-closed}
  \label{fig:sub1}
\end{subfigure}%
\begin{subfigure}{.3\textwidth}
  \centering\small
$123456$\\
$12356\quad 12346\quad 12345$\\
$1236\quad 1235\quad 1234$\\
$456\quad 236\quad 135\quad 124\quad 123$\\
$56\quad 46\quad 45\quad 23\quad 13\quad 12$\\ 
$6\quad 5\quad 4\quad 3\quad 2\quad 1$\\
$\emptyset$
  \caption{Intersection-closed}
  \label{fig:sub2}
\end{subfigure}
\caption{A union-closed family and its complement}
\label{fig:firstex}
\end{figure}

Despite its apparent simplicity  
the union-closed sets conjecture remains wide open. 
This is certainly not for lack of interest -- there 
are about 50 articles dedicated to the conjecture,
as well as several websites~\cite{ucOPG,WestWeb,ucWikipedia}. 
Due to this extensive research activity, we now know that 
the conjecture is satisfied for various union-closed families $\A$. 
For instance:
\begin{itemize} \itemsep1pt \parskip0pt 
\item if $\A$ has at most $12$ elements or at most $50$ member-sets;
\item if the number $n$ of member-sets is large compared
to the number $m$ of elements, that is, when $n\geq \tfrac{2}{3}2^m$;
\item if $n$ is small compared to $m$: when $n\leq 2m$
(where we need to assume that $\A$ is \emph{separating}, 
that is, for any two elements there exists a member-set
containing exactly one of them);
\item if $\A$ contains one of a number of subconfigurations, 
such as a singleton-set;
\item or if $\A$ has a particular structure, 
for instance, if $\A$ may be represented by a lower semimodular 
lattice, or by a subcubic graph. 
\end{itemize}
We will discuss all these results, and give proper attributions,
in the course of the article.
All these partial results notwithstanding, we still seem to be far from a
proof of the conjecture, and this is even the case for 
the obvious relaxation in which we settle for 
an element that appears in only, say, $\geq1\%$ of the member-sets. 
 The best result in this respect is 
an observation by Knill (slightly improved by W\'ojcik) 
that yields always an element of 
frequency at least $\tfrac{n-1}{\log_2n}$.

\medskip
In an article~\cite{AMG87b} of 1987, Peter Winkler\footnote{
Winkler informed us that the article was never intended 
to be published. Rather, this is the case of 
an informal letter ending up in print without Winkler even 
knowing. 
} wrote 
``the `union-closed sets conjecture' is well known indeed, except for (1) its
origin and (2) its answer!''
While the answer remains elusive, we can shed some light on its origins. 

Most authors today attribute the conjecture to Peter Frankl, 
and following Frankl~\cite{Frankl} date it to 1979. 
The sole exception are  
Balla, Bollob\'as and Eccles~\cite{BBE13}, who 
call it a ``folklore conjecture'' 
that ``was well known by the mid-1970s''.
We cannot resolve this conflict of attribution, nor do we have the intention to do so.
However, there is no doubt that Frankl did discover the conjecture
(whether he was not the first is for others to decide)
and that he played an instrumental role in popularising it. 
Consequently, we will sometimes speak of \emph{Frankl's conjecture}.

In late 1979, Frankl~\cite{FraMail} was working on traces of finite sets, 
a work that culminated in his article~\cite{Fra83} of~1983. 
Motivated by the observation that it could be used to 
improve a number of bounds, Frankl formulated the conjecture 
when travelling from Paris to Montreal. 
On his way, Frankl told the conjecture to Ron Graham, who  disseminated
it widely. 
In about 1981, Dwight Duffus learnt about it, which then led to 
its first appearance in print:
the proceedings of a workshop held in 1984 in Banff, edited by Rival~\cite{Riv85},
contain a short report of Duffus on a ``problem of P.~Frankl''.  
The second mention is Stanley~\cite{Sta86}, which simply cites Rival. 

The next time the conjecture  appeared in print, it had apparently travelled
with Franz Salzborn from Europe to Australia.
An article of 1987 in the Australian Mathematical Society Gazette~\cite{AMG87a}
reports on the Annual Meeting of the society during which Jamie Simpson
publicised the conjecture. We may only speculate that this is how 
the conjecture arrived in Papua New Guinea, 
where Renaud and Sarvate went on to write the first 
published research articles
about it~\cite{SR89,SR90,Ren91} in 1989--1991. They were succeeded in 1992
by W\'ojcik~\cite{Woj92} in Poland
and, in the USA, by Poonen~\cite{Poo92}, who wrote his influential article when he was an 
undergraduate. Many others followed in subsequent years.

\medskip
In this survey,
we aim to give a complete review of the literature on the conjecture. 
While we tried to track down every  article with a substantial connection
to the conjecture, we were not entirely successful as we could not 
obtain an unpublished manuscript of 
Zagaglia Salvi~\cite{Salvi} that, as W\'ojcik~\cite{Woj92} writes, apparently 
contains reformulations of the conjecture. 

The focus of this survey is on the methods  employed 
to attack the conjecture. Our treatment of the literature is therefore somewhat 
uneven. Whenever we can identify a technique that, to our eyes, seems
interesting and potentially powerful we discuss it in greater detail.
\section{Elementary facts and definitions}\label{sec:elementary}

We quickly settle some notation and mention the most elementary facts. 
Let $\mathcal A$ be a  family of sets. We call 
the set $U(\mathcal A):=\bigcup_{A\in\A}A$ of all the elements
that appear in some member-set of $\A$ the \emph{universe of $\A$}.
If $\A$ is union-closed then taking the complements of 
all member-sets results in a family $\mathcal D=\{U(\mathcal A)\sm A:A\in\A\}$
that is \emph{intersection-closed}: if $C,D\in\D$ then also $C\cap D\in\D$.

The union-closed sets conjecture has the following equivalent
form for inter\-section-closed families.

\newtheorem*{icsc}{Intersection-closed sets conjecture}
\begin{icsc}
Any finite intersection-closed family of at least two sets  
has an element
that is contained in at most half of the member-sets.
\end{icsc}

Continuing with notation, we denote by 
\[
\A_x:=\{A\in\A:x\in A\}.
\]
the subfamily of member-sets containing any given element $x\in U(\A)$.
The cardinality $|\A_x|$ is the \emph{frequency of $x$} in $\A$. 
We also introduce notation for the complement of $\A_x$:
\[
\A_{\overline x}:=\A\sm\A_x=\{A\in\A:x\notin A\}.
\]
We point out that, if $\A$ is union-closed, 
both $\A_x$ and $\A_{\overline x}$ are union-closed as well.

With this terminology, 
the union-closed sets conjecture  states that in every (finite) union-closed family~$\A$
there is an $x\in U(\A)$ with 
\(
|\A_x|\geq\tfrac{1}{2}|\A|. 
\)
We will call such an element $x$  \emph{abundant}. 
When we consider an intersection-closed family $\mathcal D$, 
the intersection-closed sets conjecture asserts the existence of an element $y\in U(\D)$ 
with $|\D_y|\leq\tfrac{1}{2}|\D|$. Such a $y$ is \emph{rare} in $\D$.
(We realise that this leads to the slightly bizarre situation that 
an element with frequency $|\A_x|=\tfrac{1}{2}|\A|$ is at the same 
time abundant and rare.)

As Poonen~\cite{Poo92} observed,  the union-closed sets conjecture becomes 
false if the family is allowed to have infinitely many member-sets. 
Indeed, the union-closed family consisting of the sets
$\{i,i+1,i+2,\ldots\}$ for every positive integer $i$ has infinitely many
member-sets but no element has infinite frequency. 
As a consequence, we will tacitly presuppose that every union-closed family
considered in this survey has only finitely many member-sets. 

Additionally, we will 
always require the universe to be finite as well. 
This is no restriction. If, for a union-closed family $\A$, 
the universe has infinite cardinality there 
will be infinitely many pairs of elements $x$ and $y$ in the 
universe of  $\A$ that cannot be separated by $\A$, in the sense 
that $x\in A$ if and only if $y\in A$ for all $A\in\A$.
In that case, we may simply 
delete $y$ from all member-sets of $\A$. This results again in a union-closed
family that satisfies the union-closed sets conjecture if and only 
if $\A$ does. Consequently, it suffices to prove the conjecture for 
\emph{separating} families $\A$, those in which, for any two distinct  
elements $x,y\in U(\A)$, there is an $A\in\A$ that contains exactly one of $x,y$.
It is an easy observation that the universe of any (finite) 
separating family is finite.

We remark furthermore that, if necessary, we may always assume 
a union-closed family to include the empty set as a member. Adding $\emptyset$
will at most increase the number of sets, while obviously the frequency of any
element stays the same. In the case of an intersection-closed family $\D$, 
it is no restriction to suppose that $\emptyset,U(\D)\in\D$. 
Indeed, adding $U(\D)$ to $\D$ makes satisfying the intersection-closed sets conjecture
only harder, while $\emptyset$ is always a member-set of $\D$ unless there 
is an element $x$ appearing in every set of $\D$. In that case, 
deleting $x$ from every member results in an intersection-closed family
that satisfies the conjecture if and only if $\D$ does.

Given a family $\mathcal S$ of sets, the \emph{union-closure} 
of $\mathcal S$ is the \ucf~$\A$ defined by
\[
\A = \big\{\bigcup_{S\in \mathcal S'}S : \mathcal S' \subseteq \mathcal S\big\}.
\]
We may also say that $\A$ is \emph{generated} by $\mathcal S$.

Every \ucf~$\A$ has a unique subset $\B \subseteq \A$ such 
that (a) $\A$ is the union-closure of $\B$ and (b) 
$\B$ is inclusionwise minimal with this property. 
Observe that $\B$ is simply the subfamily of 
non-empty sets $B \in \A$ with the property that 
if $B = X \cup Y$ for some $X,Y \in \A$, then $X=B$ or $Y=B$.
The sets in $\B$ are the \emph{basis sets} of $\A$.
Observe that $\A\sm\{B\}$ is union-closed for $B\in\A$
if and only if $B$ is a basis set (or $B=\emptyset$).

Finally, for $i,n\in\mathbb N$ we use the notation $[n]$ to denote 
$\{1,\ldots, n\}$ and $[i,n]$ for the set $\{i,i+1,\ldots, n\}$. 
We write $\pow{X}$ for the power set of a set $X$. 
Any set of cardinality $k$ is a \emph{$k$-set}.
For a set $X$ and an element $x$, we often write $X+x$ for $X\cup\{x\}$
and $X-x$ for $X\sm\{x\}$.

\section{The many faces of the conjecture}

The union-closed sets conjecture has several equivalent reformulations
that each highlight a different aspect.
In this section we present three reformulations, one in 
terms of lattices, one in the language of graphs and the last 
again in terms of sets. 
That the same problem can be posed quite naturally in such different fields 
is a clear indication that Frankl's question is a very basic and fundamental one.
 
The reformulations also help us to gain confidence in the veracity of the conjecture. 
Indeed, each offers natural special cases such as semimodular lattices or subcubic graphs
that would appear quite artificial in the other formulations. Proving the conjecture
for such special cases then clearly adds evidence in support of the conjecture. 
Finally, each reformulation opens up new tools and techniques to attack the conjecture.

\subsection{The lattice formulation}

Already in its earliest mention~\cite{Riv85}
it is recognised that the union-closed sets conjecture,
or rather its twin, the intersection-closed sets conjecture, has an equivalent
formulation in terms of lattices. In fact, any intersection-closed\footnote{
Or union-closed family, for that matter. However, it seems customary
in the lattice context to consider intersection-closed families. 
} 
family together with inclusion forms a lattice.

We recall a minimum of lattice terminology.
A \emph{finite lattice} is a finite poset $(L,\leq)$ in which every pair $a,b\in L$ of elements
has a unique greatest lower bound, denoted by $a\meet b$ (the \emph{meet}), and a unique
smallest upper bound, denoted by $a\join b$ (the \emph{join}). 
All the lattices considered in this survey will be finite. 
The unique minimal element is denoted by~$0$, the unique 
maximal element is~$1$.
A non-zero element $a\in L$ 
is \emph{join-irreducible} if $a=b\join c$ implies $a=b$ or $a=c$.
We write $\upcl{a}:=\{x\in L: x\geq a\}$. For more 
on lattices see, for instance, Gr\"atzer~\cite{Graetzer03}.

\medskip
Let us first see that an intersection-closed family $\A$ 
defines a lattice
in a quite direct way.
This is illustrated in Figure~\ref{firstlattice},
which shows the lattice corresponding 
to the family of Figure~\ref{fig:sub2}. 
As pointed out in the previous section, we may assume
that $\A$ contains its universe $U(\A)$.
Then $(\A,\subseteq)$ is a lattice. 
Indeed, the unique greatest lower bound of any $A,B\in\A$
is $A\meet B=A\cap B\in\A$, while $U(\A)\in\A$ guarantees that $A$ and $B$
always have a minimal upper bound. Such a minimal upper bound is unique:
If $R$ and $S$ are two upper bounds then also $R\cap S\in\A$ is an upper bound.
Let us note that while  $A\join B$ 
always contains $A\cup B$, it is usually larger. 

\begin{figure}[ht]
\centering
\includegraphics[scale=1]{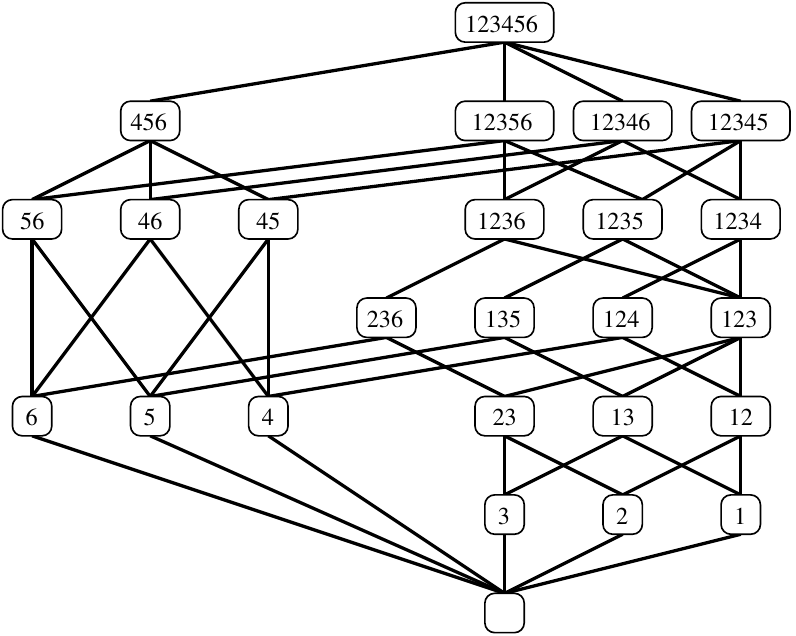}
\caption{The lattice of the set system in Figure~\ref{fig:firstex}. The
join-irreducible elements are 
precisely $\{1\},\{2\},\{3\},\{4\},\{5\},\{6\}$.}\label{firstlattice}
\end{figure}

We now state the lattice formulation of Frankl's conjecture:
\begin{conjecture}\label{latticeconjup}
Let $L$ be a finite lattice with at least two elements. Then there
is a join-irreducible element $a$ with $|\upcl a|\leq \tfrac{1}{2}|L|$.
\end{conjecture}

Let us see why Conjecture~\ref{latticeconjup} is equivalent to the intersection-closed
sets conjecture.
Let $\A$ be an intersection-closed family containing its universe and
consider the lattice $(\A,\subseteq)$.
Assume Conjecture~\ref{latticeconjup} to hold, that is, there is a join-irreducible
$J\in\A$ with $|\upcl J|\leq\tfrac{1}{2}|\A|$. 
Suppose that every element of $J$ appears in some proper subset of 
$J$ that is in $\A$:
$\bigcup_{A\subset J}A=J$. Then, $\bigvee_{A\subset J}A\supseteq \bigcup_{A\subset J}A=J$,
from which follows that $\bigvee_{A\subset J}A=J$, which is impossible as $J$ 
is join-irreducible. Thus
there is an $x\in J$ that does not lie in any proper subset of~$J$.

Next, consider an $A\in\A$ containing $x$. Then $J\cap A$ is a subset 
of $J$ containing $x$ and therefore equal to $J$.
In particular, $J\subseteq A$ and thus $A\in\upcl J$. 
Since    $|\upcl J|\leq\tfrac{1}{2}|\A|$, it follows 
that $x$ appears in at most half of the member-sets of~$\A$.

For the other direction, consider a lattice $L$ and 
associate to every $x\in L$ the set $S(x)$ of join-irreducible 
elements $z$ with $z\leq x$. Then, for $x,y\in L$ we obtain 
that $S(x\meet y)=S(x)\cap S(y)$, and thus the family $\A=\{S(x):x\in L\}$
is intersection-closed. Moreover, $|\A|=|L|$.

Supposing that the intersection-closed sets conjecture holds, we
obtain a join-irreducible $x\in L$ that is contained in at most half
of the member-sets of $\A$. Then for any $y\geq x$, it follows that 
$x\in S(y)$ and thus $|\upcl x|$ is bounded by the number of member-sets
of $\A$ containing $x$, which gives $|\upcl x|\leq\tfrac{1}{2}|L|$.

\begin{theorem}
Conjecture~\ref{latticeconjup} is equivalent to the union-closed 
sets conjecture.
\end{theorem}

In view of this equivalence we will say that a lattice 
\emph{satisfies Frankl's conjecture} if Conjecture~\ref{latticeconjup}
holds for it. To include the trivial case, we will extend
this to any lattice on less than two elements.

What are the advantages of the lattice formulation? In some sense, 
Frankl's conjecture is stripped down to its bare essential parts: the 
elements have vanished and all that counts is the inclusion relation 
between the sets. Moreover, in comparison with the 
set formulation new special cases become natural -- and
attackable. We will review them next.

\subsection{Lattice results}\label{sec:latticeresults}

The formulation of the lattice version resulted in a series
of verified special cases of Frankl's conjecture. Already in Rival~\cite{Riv85}
it is mentioned, without proof, that the conjecture holds for 
distributive and geometric lattices. This was explicitly proved by 
Poonen~\cite{Poo92}, who also extended the latter case to complemented 
lattices. 

Abe and Nakano~\cite{AN98} showed 
the conjecture for modular lattices, a case that includes
distributive lattices. This, in turn, was generalised by Reinhold~\cite{Rei00} to 
lower semimodular lattices.
We present the proof here, as it seems
to be the strongest result concerning lattice classes, and also because
the proof is nice and succinct.

Let $x<y$ be two elements of a lattice. Then $x$ is a \emph{lower cover} of $y$ 
if $x\leq z\leq y$ implies $x=z$ or $y=z$ for all elements $z$.
A lattice $L$ is \emph{lower semimodular} if 
$a\meet b$ is a lower cover of $a\in L$, 
whenever 
$b\in L$ is a lower cover
of $a\join b$.

\begin{theorem}[Reinhold~\cite{Rei00}]\label{thm:Reinhold}
Lower semimodular lattices satisfy Frankl's conjecture.
\end{theorem}

\begin{proof}
Let $L$ be a lower semimodular lattice with $|L|\geq 2$. 
If the unique largest element $1\in L$
is join-irreducible then Frankl's conjecture is trivially satisfied. 
If not, we may pick a lower cover $b\in L$ of $1$, and a join-irreducible $a\in L$
with $a\nleq b$. Then $1=a\join b$.

We claim that the function $\upcl{a}\to L\sm\upcl{a}$,
$x\mapsto x\meet b$ is an injection, which then finishes the proof. So, suppose
that there are two distinct $x,y\in\upcl{a}$ with $x\meet b=y\meet b$. 
As either $x\meet y < x$ or $x\meet y< y$, we may assume the former. 
This implies 
\begin{equation}\label{covers}
x\meet b = x\meet y\meet b\leq x\meet y < x.
\end{equation}
Now, as $L$ is lower semimodular, and as $b$ is a lower cover of $1=x\join b$, we
obtain that $x\meet b$ is a lower cover of $x$. Thus, $x\meet b=x\meet y$
by~\eqref{covers} and therefore
\[
a\leq x\meet y = x\meet b \leq b,
\] 
which contradicts our choice of $a\nleq b$.
\end{proof}

Theorem~\ref{thm:Reinhold} was also independently proved by Herrmann and Langsdorf~\cite{HL99}
and by Abe and Nakano~\cite{AN00}. In the latter article, the conjecture
is also verified for a superclass, lower quasi-semimodular lattices.

If there are lower semimodular lattices there are clearly \emph{upper 
semimodular} ones as well. However, this class seems to be much harder 
with respect to Frankl's conjecture. 
Already in Rival~\cite{Riv85} it is mentioned, without proof, that \emph{geometric lattices}
satisfy the conjecture. A proper proof was later given by Poonen~\cite{Poo92}. 
A lattice is geometric, and then 
upper semimodular, if it may be represented as the lattice of flats of a matroid.
Abe~\cite{Abe00} 
treats another subclass, the so called \emph{strong} upper semimodular
lattices. Cz\'edli and Schmidt~\cite{CS08} show the conjecture for upper 
semimodular lattices $L$ that are large, in the sense that 
$|L|>\tfrac{5}{8}2^m$ where $m$ is the number of join-irreducible 
elements; they also consider planar upper semimodular lattices.

Let us mention that it is an easy consequence of the lattice formulation
that, for any lattice $L$, Frankl's conjecture holds for $L$ or for its 
dual $L^*$, or both. (The dual lattice is obtained by reversing the order.) 
Duffus and Sands~\cite{DS99} and Abe~\cite{Abe02} derive stronger assertions for
special classes of lattices.

\medskip
We close this section with a wonderful application of Reinhold's theorem 
that was indicated to 
us by one of the anonymous referees.
The application concerns \emph{graph-generated}  
intersection-closed families. 
Let $G$ be a fixed graph. For every set $X\subseteq V(G)$ we write 
$E_X$ for the set of edges of $G$ that have both their endvertices in $X$. 
Then $\{E_X:X\subseteq V(G)\}$ is intersection-closed.

\begin{theorem}[Knill~\cite{Kni94}]
Given a graph $G=(V,E)$ with at least one edge, 
the \icf~$\{E_X:X\subseteq V\}$ satisfies the \icc.
\end{theorem}

This result is also part of Knill's PhD thesis~\cite{Kni91}.
The theorem was later restated as a conjecture by El-Zahar~\cite{EZ97}, and, as a 
response to El-Zahar's paper, 
reproved by Llano, Montellano-Ballesteros, Rivera-Campo 
and Strausz~\cite{LMRS08}.

As $L=\{E_X:X\subseteq V(G)\}$ is intersection-closed, it is a lattice 
with respect to $\subseteq$. We show that $L$ is lower semimodular.
Thus, Knill's theorem becomes a consequence of  Theorem~\ref{thm:Reinhold}. 

We call $X\subseteq V(G)$ \emph{proper} if $E_X\neq E_{X'}$
for any $X'\subsetneq X$. 
Note that $L=\{E_X:X\subseteq V(G) \mbox{ and $X$ is proper}\}$, and so we may restrict our attention to proper vertex sets.
Let $X,Y\subseteq V(G)$ be proper. 
First we note that 
\[
E_X\meet E_Y = E_X\cap E_Y = E_{X\cap Y}
\emtext{ and } E_X\join E_Y = E_{X\cup Y}.
\]
Next we observe that $E_X$ is a lower cover of $E_Y$ if and only if
\[
\emtext{
$Y=X+y_1$  or $E_Y=E_X+y_1y_2$ for some $y_1,y_2\in Y\sm X$.
}
\]
Indeed, let $E_X$ be a lower cover of $E_Y$ and consider an edge $y_1y_2\in E_Y\sm E_X$.
 Then, $E_X\subsetneq E_{X\cup\{y_1,y_2\}}
\subseteq E_Y$ and thus $Y=X\cup\{y_1,y_2\}$. Now, if one of $y_1,y_2$, $y_2$ say, is contained
in $X$ we have $Y=X+y_1$ and we are in the first case. If $y_1,y_2\notin X$ then 
neither of $y_1,y_2$ may have a neighbour in $X$ as otherwise $E_X$ would be 
a proper subset of $E_{X+y_1}$ or of $E_{X+y_2}$. The other direction is obvious.

So, assume that for proper $A,B\subseteq V(G)$, the set $E_B$ is a lower cover of $E_A\join E_B$.
Then there are $a_1,a_2\in A\sm B$ so that either $A\cup B=B+a_1$ or $E_{A\cup B}=E_B+a_1a_2$.
If $A\cup B=B+a_1$ then $A=(A\cap B)+a_1$, and $E_{A\cap B}$ is a lower cover of $E_A$.
In the other case, when $E_{A\cup B}=E_B+a_1a_2$ we get
\[
E_A=E_A\cap E_{A\cup B} = (E_A\cap E_B)+a_1a_2 = E_{A\cap B}+a_1a_2,
\]
and again $E_{A\cap B}$ is a lower cover of $E_A$. Thus, $L$ is lower semimodular,
and Knill's theorem is proved.

\medskip
El-Zahar~\cite{EZ97} observed that, 
when Knill's theorem is generalised to hypergraphs, 
it becomes yet another reformulation of the union-closed sets conjecture. 

\subsection{The graph formulation}

A more recent  reformulation of the union-closed sets conjecture
is stated in terms of maximal stable sets of bipartite graphs.
A \emph{stable set} of a graph $G$ is a vertex subset 
so that no two of its vertices are adjacent.
A stable set is called \emph{maximal} if no further vertex of $G$ can be added 
without violating the stable set condition. We refer to Diestel~\cite{diestelBook05}
for general terminology and notions on graphs.

The graph formulation of the union-closed sets conjecture 
is as follows:

\begin{conjecture}\label{GraphFranklConj}
Any bipartite graph with at least one edge contains in each of its 
bipartition classes a vertex that lies in at most half of the maximal
stable sets.
\end{conjecture}

The conjecture was proposed by Bruhn, Charbit, Schaudt and Telle~\cite{BCT12},
who also proved the equivalence to Frankl's conjecture.
In analogy to the \icc, let us call a vertex \emph{rare} if it is contained in at most half of the maximal stable sets.
Note that for every edge $uv$ of a bipartite graph, always one of $u$ and $v$ is rare.
Indeed, this follows directly from the fact that no stable set
may contain both $u$ and $v$.
Hence, in a hypothetical counterexample to Conjecture~\ref{GraphFranklConj}, one 
bipartition class
of the graph contains only rare vertices, 
while no vertex in the other class is rare.

We sketch why Conjecture~\ref{GraphFranklConj} and the \icc\ are equivalent.
\begin{theorem}\selfcite{BCT12}
Conjecture~\ref{GraphFranklConj} holds if and only if the union-closed sets conjecture
is true.
\end{theorem}
\begin{proof}
To prove equivalence to the intersection-closed sets conjecture, 
let us first consider a bipartite graph $G$ 
with bipartition classes $X,Y$. By symmetry it is enough to find a rare vertex 
in $X$. 
Let $\A$ be the set of maximal stable sets of $G$.
It is straightforward to check that the traces of maximal stable sets in $X$, 
the set $\{A \cap X : A \in \A\}$, is intersection-closed.  
Thus, if the intersection-closed sets conjecture is true, 
there must be a rare element $x$ of $\{A \cap X : A \in \A\}$,
which then is a rare vertex of $G$.

For the converse direction, let an intersection-closed family $\A$ 
be given. We may assume that $\A$ contains its universe $U$.
We define a bipartite graph $G=(V,E)$ on $V=\A\cup U$
with edge set $E = \{ Sx : S \in \A,\, x \in U,\, x \in S \}$.
That is, $G$ is the incidence graph of $\A$. See Figure~\ref{firstgraph}
for an illustration.

Then, if $\B$ denotes the set of maximal stable sets of $G$,
it follows that $\A = \{B \cap U : B \in \B\}$.
Thus, if $x$ is a rare vertex of $G$ in $U$, then $x$ is a rare element of $\A$.
This completes the proof.
\end{proof}

\begin{figure}[ht]
\centering
\includegraphics[scale=0.8]{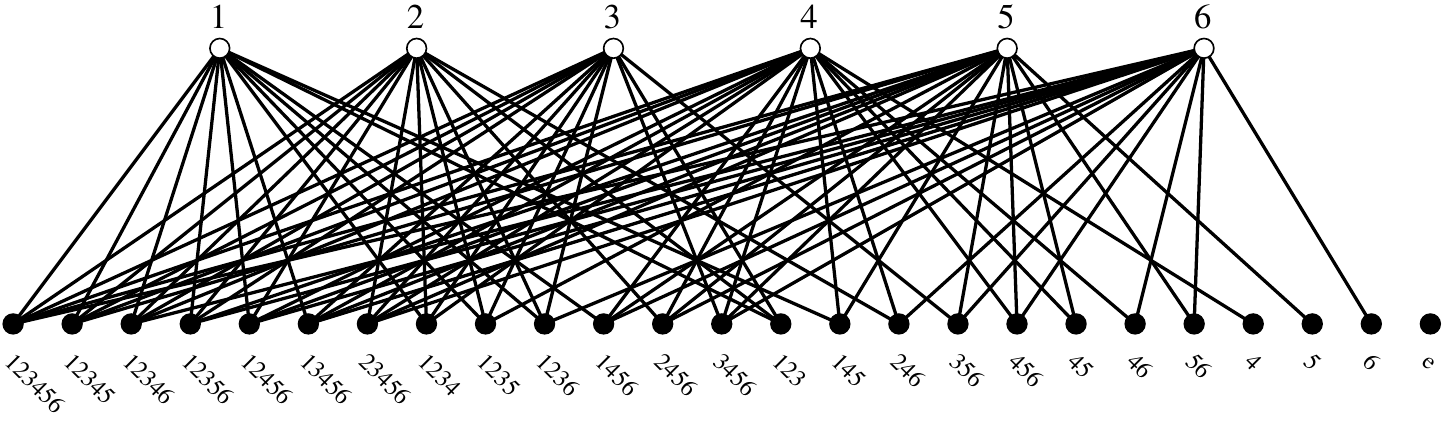}
\caption{The incidence graph of the \icf~shown in Figure~\ref{fig:firstex}}\label{firstgraph}
\end{figure}

As for the lattice 
fromulation, we will say that a bipartite graph \emph{satisfies Frankl's conjecture}
if the graph is not a counterexample to Conjecture~\ref{GraphFranklConj},
or if it is edgeless.

Figure~\ref{firstgraph} shows the graph representation of intersection-closed
family in Figure~\ref{fig:firstex}. 
We have to admit that it does not appear very appealing, as
listing the family seems much simpler.
Nonetheless, the graph formulation allows for a very compact representation 
of Frankl's conjecture. This is exemplified by the graph in Figure~\ref{secondgraph}
that encodes the same family as the graph in Figure~\ref{firstgraph}.
We arrive at this graph by iteratively deleting any vertex $v$
whose neighbourhood is equal to the union of neighbourhoods 
of some other vertices. It is easy to check
that the resulting  graph with $v$ deleted satisfies the conjecture only if
the original graph does, see also \cite{BCT12}.

\begin{figure}[ht]
\centering
\includegraphics[scale=0.8]{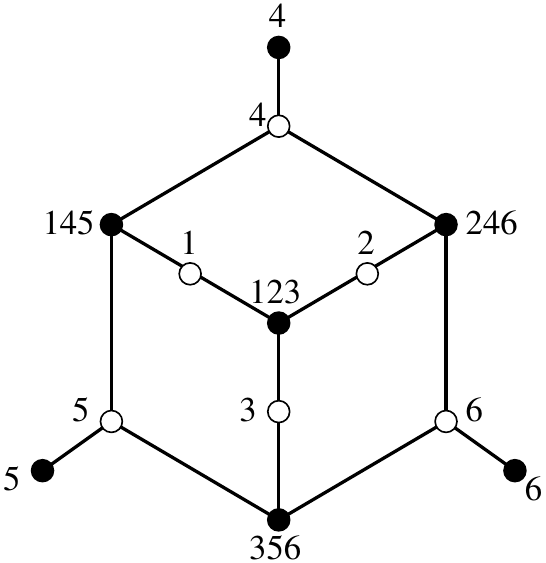}
\caption{A more succinct representation}\label{secondgraph}
\end{figure}

\subsection{Graph results}\label{sec:GraphResults}

The literature on graphs provides a rich selection of 
natural graph classes, even bipartite ones, that may now serve as test cases
for Frankl's conjecture. 
So far, the conjecture has been verified for 
chordal bipartite, subcubic, series-parallel~\cite{BCT12} and, in an approximate version, random bipartite graphs~\cite{BS12}. 
We present some of these results here.

\medskip
A bipartite graph is said to be \emph{chordal bipartite} if 
deleting vertices from the graph can never result in a chordless cycle
of length $\geq 6$.

\begin{theorem}\selfcite{BCT12}\label{thm:BCT}
Chordal bipartite graphs satisfy Frankl's Conjecture.\sloppy
\end{theorem}

The proof rests on the local structure of chordal
bipartite graphs. 
This is a general strategy that we will discuss 
in more detail in Section~\ref{localconfigs}.
The main tool here is the following lemma, where we denote by $N^2(x)$ the neighbours of the neighbours of a vertex $x$ (including $x$).

\begin{lemma}\selfcite{BCT12}\label{lem:NeighbourhoodRare}
Let $x,y$ be two adjacent vertices of a bipartite graph with $N^2(x) \subseteq N(y)$.
Then $y$ is rare.
\end{lemma}

\begin{proof}
Let $\A$ denote the maximal stable sets of the chordal bipartite graph $G$,
and consider $A\in\A_y$, that is, a maximal stable set containing $y$. 
Since $y\in A$, no neighbour of $y$ may be in $A$ and hence $N^2(x)\cap A=\emptyset$
as $N^2(x) \subseteq N(y)$. Therefore, no vertex in $N(x)$ is adjacent with a
vertex in $A$, which implies $N(x)\subseteq A$.

We  now construct an injective mapping $\A_y\to\A_x$:
given a set $A \in \A_{y}$, first remove all members of $N(x)$ from $A$ 
and then fill up the resulting set to a maximal stable set with vertices from $N^2(x)$.
Finally, since $x$ is adjacent to~$y$, we have $\A_x \subseteq \A_{\overline{y}}$.
Altogether, there is an injection $\A_y \to \A_{\overline{y}}$, which means that $y$ is rare.
\end{proof}

To finish the proof of Theorem~\ref{thm:BCT} it now suffices to observe that 
a type of vertex known as a \emph{weakly simplicical} vertex satisfies the conditions
of the lemma. That such a vertex always exists 
in each bipartition class is known from the literature on chordal bipartite graphs. 
For details see~\cite{BCT12}.

\medskip
Using results of Vaughan on $3$-sets and 
Knill's graph generated families (discussed in Sections~\ref{localconfigs} and~\ref{sec:latticeresults} respectively), we obtain Frankl's conjecture for 
another natural  graph class. Recall that a graph is \emph{subcubic} 
if every vertex has degree at most three. 
\begin{theorem}\selfcite{BCT12}
Every subcubic bipartite graph satisfies Frankl's conjecture. 
\end{theorem}

The third class of graphs we treat are random bipartite graphs, 
where we can only prove a slight weakening of Frankl's conjecture.
A \emph{random bipartite graph} is a graph on bipartition classes of cardinalities $m$ and $n$, 
where any two vertices from different classes are independently joined by an 
edge with probability $p$.

For $\delta>0$,
let us say that a bipartite graph \emph{satisfies \fc~up to~$\delta$}
if 
each of its two bipartition classes has a vertex
for which the number of maximal stable sets containing it is at most
$\tfrac{1}{2}+\delta$ times the total number 
of maximal stable sets. 
We say that \textit{almost every} random bipartite graph 
has property $P$ if 
for every $\varepsilon >0$ there is an $N$ such that, whenever $m+n \ge N$, 
the probability that a random bipartite graph on $m+n$ vertices 
has $P$ is at least $1-\varepsilon$.

\begin{theorem}\selfcite{BS12}
Let $p\in (0,1)$ be a  fixed edge-probability.
For every $\delta>0$, almost every random bipartite graph satisfies \fc~up to~$\delta$.
\end{theorem}

The main tool in the proof is the averaging approach detailed in Section~\ref{sec:averaging}.

\subsection{The Salzborn formulation}

Returning to the sets point of view, let us present 
a surprising reformulation of the conjecture that W\'ojcik~\cite{Woj99}
attributes to Salzborn~\cite{Sal89}. Recall that a union-closed family $\A$
is separating if for any two elements of its universe there is
a member-set that contains exactly one of the two. 
It is easy to check that $\A$ needs to have at least $|U(\A)|$ non-empty sets 
to separate all elements of its universe. Thus, if $\emptyset\in\A$
then $\A$ will have at least $|U(\A)|+1$ member-sets. It turns out that the families 
with this minimum number of member-sets have a surprisingly rich structure. 

Let us call a union-closed family $\mathcal N$ \emph{normalised} 
if it holds that $\emptyset\in\mathcal N$, $\mathcal N$ is separating 
and $|U(\mathcal N)|=|\mathcal N|-1$.
The following conjecture may be found in W\'ojcik~\cite{Woj99}, 
or, with less details, in Salzborn~\cite{Sal89}.

\begin{conjecture}[Salzborn~\cite{Sal89}]\label{salzconj}
Any normalised family $\mathcal N\neq\{\emptyset\}$ contains a basis set $B$
of size $|B|\geq\tfrac{1}{2}|\mathcal N|$.
\end{conjecture}

Following W\'ojcik~\cite{Woj99}, we outline
 why Salzborn's conjecture implies the \ucc.
Consider a union-closed family $\A$ that we may assume to contain $\emptyset$
as a member-set. 
We define 
\[
\A\nsub{X}:=\{A\in\A:A\nsubseteq X\}\text{ and }\A^*:=\{\A\nsub{X}:X\in\A\}.
\]

It is easy to check that $\A^*$ is  union-closed and  separating.
We note that 
$X\subseteq Y$ if and only if
$\A\nsub{X}\supseteq\A\nsub{Y}$ for any $X,Y\in\A$.
This has several consequences. Firstly,  $\A\nsub{X}\neq\A\nsub{Y}$
if $X\neq Y$, which implies that $|\A|=|\A^*|$. 
Secondly, $U(\A^*)=\A\nsub{\emptyset}=\A\sm\{\emptyset\}$.
Finally, we remark that $\A^*$ has the dual lattice structure of $\A$.

To summarise, $\A^*$ is normalised and has the same number of members as~$\A$.
Next, we consider the basis sets of $\A^*$. 
\begin{equation}\label{salzbasis}
\text{\em
Every basis set of $\A^*$ is of the form $\A_x$ for some $x\in U(\A)$.
}
\end{equation}
Indeed, consider a basis set $\A\nsub{X}$ of $\A^*$, and observe that 
$\A\nsub{X}=\bigcup_{y\in U(\A)\sm X}\A_y$.
Pick a smallest set $S\subseteq  U(\A)\sm X$ so that still 
$\A\nsub{X}=\bigcup_{y\in S}\A_y$ and consider a bipartition $S_1\cup S_2=S$.
Since $\A_y=\A\nsub{U(\A_{\overline y})}$, 
both $\bigcup_{y\in S_1}\A_y$ and $\bigcup_{y\in S_2}\A_y$
are members of $\A^*$.
Since $\A\nsub{X}=\bigcup_{y\in S_1}\A_y\cup \bigcup_{y\in S_2}\A_y$,
$\A\nsub{X}$ is the union of two member-sets of $\A^*$. 
As $\A\nsub{X}$ is a basis set that implies that already 
$\A\nsub{X}=\bigcup_{y\in S_i}\A_y$ for $i=1$ or $i=2$,
which by the minimality of $S$ forces $S=S_i$. Therefore, 
$S$ has to contain a unique element $x$, that is, 
$\A\nsub{X}=\A_x$. 

\medskip
Assume now Conjecture~\ref{salzconj} to hold. Then the normalised
family $\A^*$ contains a basis set $B^*$ with 
\[
|B^*|\geq\tfrac{1}{2}|\A^*|=\tfrac{1}{2}|\A|.
\]
As $B^*=\A_x$ for some $x\in U(\A)$ by~\eqref{salzbasis} we 
deduce that $\A$ satisfies the \ucc. 
We therefore have proved one direction of:
\begin{theorem}[Salzborn~\cite{Sal89}]
Conjecture~\ref{salzconj} is equivalent to the \ucc.
\end{theorem}
We omit the proof of the other direction, which may be found 
in W\'ojcik~\cite{Woj99}.

\medskip
Why do we find the Salzborn reformulation surprising? 
At first glance, normalised families seem to be very restricted
and in some sense this is true. For instance, the statement of the \ucc\
is almost trivial for them, see Theorem~\ref{falgasbound2}.
From a lattice point of view, however, normalised families 
turn out to be as general as union-closed families. We have already remarked
that $\A^*$ has the dual lattice structure of $\A$, which directly 
implies that every lattice type of a union-closed family is realisable
as a normalised family. 

We know only one application of the Salzborn formulation: 
W\'ojcik~\cite{Woj99} uses it to obtain a non-trivial lower 
bound on the maximum frequency of an element in a union-closed 
family; see the next section.

The family $\A^*$ also appears in Johnson and Vaughan~\cite{JV99},
although defined in a slightly different way.
In order to 
obtain a duality result,
 Johnson and Vaughan associate to any union-closed family $\A$ the 
dual family $\A^*$ and then observe that the 
\ucc\ is satisfied for at least one of $\A$ and $\A^*$. 
We note that the analogous results in the lattice formulation and in the 
graph formulation are almost trivial: for lattices this amounts
to considering the dual lattice, and for graphs it reduces
to the observation that no stable set may contain both endvertices of an edge.

\medskip
The majority of the results on the union-closed sets conjecture are with respect 
to the original set formulation. In the remainder of this article we stick to this 
formulation as well. However, a good part of the discussed techniques has a more
or less direct analogue in the other formulations.

\section{Obstacles to a proof}\label{sec:hardness}

There are many results on special cases of the conjecture. Amazingly, if we consider
an arbitrary union-closed family, without any special structure or information 
on the number of elements, 
(almost) the best result we have seems to be a simple observation due
to Knill:
\begin{theorem}[Knill~\cite{Kni94}]
Any union-closed family $\A$ on $n$ member-sets
has an element of frequency at least $\tfrac{n-1}{\log_2(n)}$.
\end{theorem}
\begin{proof}
We may assume that $\emptyset\in\A$.
Let us choose $S\subseteq U(\A)$ minimal such that every non-empty set of $\A$ 
  intersects $S$.
Then for every $x\in S$ there is a $A\in\A$ with $A\cap S=\{x\}$; otherwise
$S-x$ would still meet every non-empty $A\in\A$, which contradicts
 the minimality of $S$.
As $\A$ is union-closed it follows that $\{A \cap S : A \in \A\} = 2^{S}$.
Hence $n \ge 2^{|S|}$ and so $|S| \le \log_2(n)$.
As every of the $n-1$ non-empty member-sets of $\A$ intersects $S$,
there is an element in $S$ that belongs to at least $(n-1) / \log_2(n)$ many member-sets of $\A$.
\end{proof}
W\'ojcik~\cite{Woj99} improved the bound to 
$\frac{2.4n}{\log_2{n}}$ for large $n$. His proof is not 
trivial, but the result is still far from Frankl's conjecture.

\medskip
Here are two observations that could be interpreted as signs
that the conjecture is, after all, perhaps not as hard as thought:
normally the most frequent element appears more often than needed, 
and there are several abundant elements.
Indeed, the powerful averaging technique discussed in Section~\ref{sec:averaging} builds solely on these facts.

These observations are due to Poonen, who also found 
exceptions to them.
Power sets are 
an obvious example for  families in which 
the maximum frequency is exactly half the size of the family. Poonen 
conjectured that, among separating families, these are the only ones.
\begin{conjecture}[Poonen~\cite{Poo92}]\label{powconj}
Let $\A$ be a separating union-closed family. Unless $\A$ is a power set,
it contains an element that appears in strictly more than half of the member-sets
of $\A$. 
\end{conjecture}
A similar conjecture was offered by Renaud~\cite{Ren91}.
Moreover,
Poonen described families with a unique abundant 
element and again conjectured that these are the only ones:

\begin{conjecture}[Poonen~\cite{Poo92}]\label{manyconj}
Let $\A$ be a separating union-closed family on universe $U$.
If $\A$ contains a unique abundant element $a$ then 
\[
\A = \{\emptyset\}\cup \{B+a:B\subseteq U-a\}.
\]
\end{conjecture}

If these conjectures are to be believed, then there is a bit 
of a margin when attacking the union-closed sets conjecture. 
So, why then has the conjecture withstood more than twenty years
of proof attempts? 

\medskip

The obvious first approach is to try an induction, for instance 
on the number of member-sets. If, given a union-closed family, 
we could delete one (or two) basis sets so that the maximum frequency
drops then, by induction, the original family would satisfy 
the conjecture, too. Unfortunately, this is not always possible:
in a power set of sufficient size, 
deleting one or two basis sets will never reduce
the maximum frequency. 

So, naive induction will not succeed. Often, induction can only
be made to work if the hypothesis is strengthened, usually 
by exploiting some structural insight. However, we feel that we 
are lacking in just that. We do not know
what the extremal 
families look like, those that have minimal maximum frequency
among all union-closed families of a given size. 
So far, there are not even any good candidates. We will continue this
discussion in Section~\ref{exConway}. 

A second reason why the conjecture has resisted so long
lies in the weakness of the techniques at our disposal. 
Let us briefly review the main techniques
used to prove that a given family satisfies the conjecture: 
\emph{injections}, \emph{local configurations} and \emph{averaging}. 
In averaging we try to show that the average frequency is large 
enough so that some element must be abundant. Averaging is very 
powerful but has the drawback that there are families for which the 
average is simply too low for the method to work. We discuss averaging and its
limits in Section~\ref{sec:averaging}. For the local configurations method
one strives to identify small families so that any large union-closed
family containing the small one will automatically satisfy the conjecture. 
Unfortunately, given what we know at the moment it seems doubtful 
that we will be able to show that any union-closed
family always contains such a local configuration. We will 
have a closer look at local configurations in the next section.

That leaves injections, the simplest of the three techniques. 
For an almost trivial example, consider the case when a union-closed
family $\A$ contains a singleton, that is, there is an element
$x$ so that $\{x\}\in\A$. Then 
\[
\A_{\overline x}\to\A_x,\,A\mapsto A+x
\]
defines an injection, which clearly implies that 
$2|\A_{x}|\geq |\A_x|+|\A_{\overline x}|=|\A|$.
Consequently, $x$ is abundant. In fact, we have used this method
already twice: once for lower semimodular lattices and 
then for chordal bipartite graphs. The main problem with the 
injection method is that we need to first identify an element that is 
likely to be abundant. 

Sarvate and Renaud~\cite{SR89} were probably the first to observe (in print) that 
a singleton is always abundant. In a similar way, 
one of the two elements of any $2$-set is abundant. 
The pattern, however, breaks with $3$-sets. 
Renaud and Sarvate~\cite{SR90}
describe a family with a unique smallest member-set
of $3$ elements, none of which is abundant. 
Poonen~\cite{Poo92} constructs a similar family,
a generalisation of which we present here:

For each $k \ge 3$ we define a \ucf~$\A^k$ with the property that $[k]$ is the unique smallest set, but no element of $[k]$ is abundant. 
For this, we use Poonen's notation $\A \uplus \B$ for two set families $\A$ and $\B$
to denote the family
\[
\A \uplus \B:=\{S \cup T : S\in\A,\,T\in\B\}.
\]

Now let
\begin{align*}
\A^k = \{[k]\} \cup \bigcup_{i=1}^k(\{\emptyset,\{i\},[k]\}\uplus\B^i) \cup (2^{[k]}\uplus[k+1,3k]),
\end{align*}
where
\[
\B^i=\{[k+1,3k]\sm\{2i+2\},[k+1,3k]\sm\{2i+3\}\}\text{ for every }i \in [k].
\]
Note that the set $[k]$ is the unique smallest set in $\A$.
In total, $\A^k$ contains $1 + 6k + 2^k$ many sets, but every $i \in [k]$ is contained in exactly $1+(2k+2)+2^{k-1}$ sets of $\A$.
Therefore, no element of $[k]$ is abundant.

\medskip
Poonen's family highlights one of the major obstacles on the 
way to a proof of the union-closed sets conjecture: we do not 
know where to expect an abundant element. 
However, 
there are special cases where this is known. 
We treat these cases next.

\section{Local configurations}\label{localconfigs}

Sarvate and Renaud~\cite{SR89} observed that
any singleton in a union-closed family is abundant, and of the two 
elements of a $2$-set at least one is abundant. 
This motivates the search for good \emph{local configurations}: 
a family $\mathcal L$ on few elements so that any union-closed family
$\A$ containing $\mathcal L$ has an abundant element among the elements of $\mathcal L$.
Poonen~\cite{Poo92} gives a complete characterisation of such families:

\begin{theorem}[Poonen~\cite{Poo92}]\label{poonenlocal}
Let $\mathcal L$ be a \ucf~with universe $[k]$. 
The following statements are equivalent:
\begin{enumerate}[\rm (i)]
\item
Every union-closed family $\A$ 
containing $\mathcal L$ satisfies the union-closed sets conjecture.
In particular, $\A$ has an abundant element in $[k]$.
\item
There are reals $c_1, c_2 , \ldots, c_k\geq 0$ with $\sum_{i=1}^kc_i=1$ such that 
for every \ucf~$\mathcal K \subseteq 2^{[k]}$ with $\mathcal K = \mathcal L \uplus \mathcal K$ 
it holds that
\[
\sum_{i=1}^k c_i |\mathcal K_i| \ge \tfrac{1}{2}|\mathcal K|.
\]
\end{enumerate}
\end{theorem}

We stress that~(ii) is indeed a local condition: for fixed $k$ there are only finitely
many such families~$\mathcal K$.
As an application of his theorem, 
Poonen showed that the \ucf~consisting of a $4$-set together with any three 
distinct $3$-subsets satisfies the conditions of his theorem.
This was later generalised by Vaughan~\cite{Vaug04}
to three distinct $3$-sets with a non-empty common intersection.
As mentioned in Section~\ref{sec:GraphResults}, Vaughan's result is 
used to prove Frankl's conjecture for subcubic bipartite graphs.

A \ucf~$\mathcal L$ as in Theorem~\ref{poonenlocal} 
is called \emph{Frankl-complete} by Vaughan~\cite{Vaug02}, \emph{FC} for short.
Several FC-families are listed in~\cite{Vaug02}, for example a $5$-set together with all 
its $4$-subsets or a $6$-set with all $5$-subsets and eight $4$-subsets.
The list was later extended by Morris~\cite{Mor04},
who, in particular, completely characterised the FC-families on at most $5$ elements.

To study FC-families in a more quantitative way, 
Morris~\cite{Mor04} introduced the function $\mbox{FC}(k,m)$ defined as the smallest $r$ 
for which 
the set of every $r$ of the $k$-sets in $[m]$ generates an FC-family.
He showed  that $\lfloor \tfrac{m}{2} \rfloor +1 \le \mbox{FC}(3,m)$, while
Vaughan~\cite{Vaug04} gave an upper bound of $\mbox{FC}(3,m) \le \tfrac{2m}{3}$.
A proof of Morris' conjecture that $\mbox{FC}(3,m) = \lfloor \tfrac{m}{2} \rfloor +1$ was announced by Vaughan~\cite{Vaug05}, 
but has apparently never been published.

Mari\'c, \v Zivkovi\'c and Vu\v ckovi\'c~\cite{MZV12} verified some known FC-families 
and found a new one using the automatic proof assistant Isabelle/HOL.
For this, they formalised the condition of FC-families to enable a computer search.
As a result, we know now
that all families containing four $3$-subsets of a $7$-set are FC-families.

\subsection{Small finite families}\label{sec:FirstFiniteCases}

The union-closed sets conjecture has been verified for  families
on few member-sets or few elements.
The current best results use local configurations
to reduce the number of special cases substantially.

With respect to the size of the universe, the conjecture has to-date
been verified up to $m=12$:
\begin{theorem}[\v Zivkovi\'c and Vu\v ckovi\'c~\cite{ZV12}]\label{m=12}
The union-closed sets conjecture holds for \ucfs~on at most $12$ elements.
\end{theorem}

The following result, that has not been improved upon
in the last twenty years, allows to leverage bounds on the universe size
to bounds on the number of member-sets:

\begin{lemma}[Lo Faro~\cite{LF94b}]\label{4q-1}
Under the assumption that the union-closed sets conjecture fails, let $m$ denote 
the minimum cardinality of $|U(\A)|$ taken over all counterexamples $\A$ to the union-closed sets conjecture.
Then any counterexample has at least $4m-1$ member-sets.
\end{lemma}

The lemma was later rediscovered by Roberts and Simpson~\cite{RSi10}.
Together with Theorem~\ref{m=12} we obtain:

\begin{corollary}\label{n=50}
The union-closed sets conjecture holds for \ucfs~with at most 50 sets.
\end{corollary}

Various authors verified the conjecture for small values of $n$ and $m$,
where as usual $n$ is the number of member-sets and $m$ the size of the universe. 
The first were Sarvate and Renaud~\cite{SR89} who treated a close 
variant  that excludes the empty set. In a first paper 
they covered all cases 
up to $n\leq 11$; in  
Sarvate and Renaud~\cite{SR90} the case analysis was pushed up to $n\leq 19$.
Using his Theorem~\ref{poonenlocal}, Poonen 
improved the bounds to  $m \le 7$ and $n \le 28$.
This was followed by
Lo Faro~\cite{LF94b}, who settled the union-closed sets conjecture for $m \le 9$ and $n \le 36$.
For this, he investigated several necessary conditions on a minimal counterexample,
among them Lemma~\ref{4q-1} above. Roberts~\cite{Rob92} shows the conjecture
up to $n\leq 40$.

Using the list of known FC-families, Morris~\cite{Mor04} proved the union-closed sets conjecture for families with $m \le 9$ and $n \le 36$, apparently unaware of the older result by Lo Faro~\cite{LF94b}.
Nevertheless, there is merit in Morris' proof as 
it showcases how FC-families may be used to substantially reduce the number of cases.
This method is at the heart of all subsequent work in this direction.

In order to prove the conjecture for $m \le 10$, 
Markovi\'c~\cite{Mar07} 
imitated the method of Theorem~\ref{poonenlocal}:
he assigns non-negative weights to the elements of $\A$ and extends this to the member-sets of $\A$.
He then observes  that a total weight of the member-sets of at least 
$\tfrac{1}{2} n$ times the weight of the universe is sufficient for the union-closed sets conjecture.
As a by-product of this method, Markovi\'c discovered a number 
of new FC-families.

Bo\v{s}njak and Markovi\'c~\cite{BM08} improve upon~\cite{Mar07}
by developing more general local configurations that allow them to verify the 
conjecture up to $m=11$. 
With a very similar method and the use of a computer, 
\v Zivkovi\'c and Vu\v ckovi\'c~\cite{ZV12} pushed this to $m \le 12$.

\section{Averaging}\label{sec:averaging}

Obviously, a union-closed family $\A$ has an element of 
frequency $\geq\tfrac{1}{2}|\A|$ if
the \emph{average frequency} is 
at least $\tfrac{1}{2}|\A|$. In other words, if
\begin{equation}\label{avfrequency}
\frac{1}{|U(\A)|} \cdot \sum_{u\in U(\A)}|\A_u|\geq \frac{1}{2}|\A|,
\end{equation}
then $\A$ satisfies the union-closed sets conjecture. 

So far, not much is gained. 
Calculating $\sum_{u\in U(\A)}|\A_u|$ directly 
is clearly out of question, as this would presuppose 
knowledge about the individual frequencies $|\A_u|$.
Fortunately, this is not necessary, as the sum of frequencies
can be determined indirectly with a
simple double-counting argument:
\begin{equation}
\sum_{u\in U(\A)}|\A_u| = \sum_{A\in\A}|A|.
\label{eq:double-counting}
\end{equation}
This identity is the heart of the averaging method. The total 
set size is usually much easier to control, and in some cases
may be estimated quite well.

Combining~\eqref{avfrequency} and~\eqref{eq:double-counting}, 
a  condition equivalent to~\eqref{avfrequency} is that 
\[
\frac{1}{|\A|} \cdot \sum_{A\in\A}|A| \ge \frac{1}{2}|U(\A)|.
\]
That is, if the \emph{average set size} of $\A$ is at least half the size 
of the universe then $\A$ again satisfies the union-closed sets conjecture.

As discussed in Section~\ref{sec:hardness}, it is not obvious where 
to look for an abundant element. 
The averaging method has the clear advantage that 
it simply sidesteps this  obstacle. 
In this section we describe how both~\eqref{avfrequency} and~\eqref{eq:double-counting} 
lead to some of the strongest results on the union-closed sets conjecture.

\subsection{Large families}

In a clearly overlooked paper, Nishimura and Takahashi~\cite{NT96} prove for the first 
time that the union-closed sets conjecture always holds for large families. 
Their proof uses the average set size argument: 
it is shown that the average set size is  greater than $\tfrac{m}{2}$, 
which implies that there is an abundant element. 

\begin{theorem}[Nishimura and Takahashi~\cite{NT96}]
Let $\A$ be a \ucf~of more than $2^m-\tfrac{1}{2}\sqrt{2^m}$ member-sets
on a universe of size $m$. Then $\A$ satisfies the union-closed sets conjecture. 
\end{theorem}

\begin{proof}
Suppose there is a set $S \subseteq U(\A)$ with $S \notin \A$ 
but $|S| \ge \frac{m}{2}$.
Then for any subset $R \subseteq S$ with $R \in \A$ it holds that $S \setminus R \notin \A$.
Thus, at least half of the subsets of $S$ are missing in $\A$.
This gives $|\A| \le 2^m - \frac{1}{2} \cdot 2^{\frac{m}{2}}$, a contradiction.
Hence, every set $S \subseteq U(\A)$ of size at least $\tfrac{m}{2}$ is contained in $\A$.
This means that the average set size is at least $\tfrac{m}{2}$, finishing the proof.
\end{proof}

Cz\'edli~\cite{Cze09} employed some involved lattice-theoretic arguments 
to push the bound from 
$2^m-\tfrac{1}{2}\sqrt{2^m}$ to $2^m-\sqrt{2^m}$.
A weaker result  than Nishimura and Takahashi's was proved by Gao and Yu~\cite{GY98}.
Recently, a serious improvement of the above bound was given by Balla, Bollob\'as and Eccles~\cite{BBE13}, 
which we present in Section~\ref{secupcompression}.

\subsection{Bounds on the average}\label{secbounds}

Averaging does not always work. It is easy to construct union closed 
families with  an average frequency and average set size that is too 
low to deduce the union-closed sets conjecture. 
Reimer~\cite{Rei03} gave a bound on the average set size that
is in some respect best possible. 

\begin{theorem}[Reimer~\cite{Rei03}]\label{Reimertheorem}
Let $\A$ be a \ucf~on $n$ sets. Then
\begin{equation}
\frac{1}{n} \cdot \sum_{A \in \A} |A| \ge \frac{\log_2 n}{2}.
\label{eq:Reimer}
\end{equation}
\end{theorem}

The result is too weak for Frankl's conjecture
as usually $\log_2(n)<m$. 
In terms of the average frequency, Reimer's bound reads as
\begin{equation}\label{Reimeravg}
\frac{1}{m} \cdot \sum_{u\in U(\A)}|\A_u| \ge \frac{\log_2 n}{m} \cdot \frac{n}{2}.
\end{equation}
We discuss the beautiful proof of Theorem~\ref{Reimertheorem} in Section~\ref{secupcompression}.

We now focus on separating \ucfs, 
where for every two elements there is a set
containing exactly one of them.
As explained in Section~\ref{sec:elementary}, for the purpose of the union-closed sets conjecture it is not a restriction to consider only separating families.

\begin{theorem}[Falgas-Ravry~\cite{FR11}]\label{thm:Ravry}
Let $\A$ be a separating \ucf~on $m$ elements.
Then 
\begin{equation}\label{falgasbound}
\frac{1}{m} \cdot \sum_{u\in U(\A)}|\A_u| \ge \frac{m+1}{2}.
\end{equation}
\end{theorem}

He remarks that this bound is stronger than Reimer's bound if $m > \sqrt{n \log_2 n}$.
The proof of~\eqref{falgasbound} is rather simple:

\begin{proof}
Assume that the elements $1 , 2 , \ldots , m$ of $U(\A)$ are labelled
in order of increasing frequency.
As $\A$ is separating, this ordering ensures
that for any $1 \le i < j \le m$ there is a set $X_{ij} \in \A$ 
such that $i \notin X_{ij}$ and $j \in X_{ij}$.
For all $1 \le i \le m-1$ let $X_i = \bigcup_{j= i+1}^m X_{ij}$, 
and put $X_0:=U(\A)$.
Observe that (a) the $X_i$ are all distinct and that (b) $[i+1,m] \subseteq X_i$.
Thus, the statement follows from
\[
\sum_{u\in U(\A)}|\A_u| \stackrel{(a)}{\ge} 
\sum_{i=0}^{m-1} |X_i| \stackrel{(b)}{\ge} \sum_{i=0}^{m-1} (m-i) = \frac{m(m+1)}{2}.
\]
\end{proof}

Let us point out an easy consequence of the proof.
As Nishimura and Takahashi observed, the union-closed sets conjecture 
holds for families that are very large with respect to their universe. 
Here we obtain the analogous result for very \emph{small} families: 

\begin{theorem}\label{falgasbound2}
Any separating family on $m$ elements with at most $2m$ member-sets
satisfies the union-closed sets conjecture.
\end{theorem}
\begin{proof}
Each of the $m$ sets $X_i$ as constructed above 
contains the most 
frequent element $x_m$.
\end{proof}
We note that this is a weaker bound than the one obtained by Lo~Faro
for a minimal counterexample (Lemma~\ref{4q-1}): $n\leq 4m-1$. 
However, Lo~Faro's techniques
do not extend easily to small families and there is a good reason for this. 
If the factor in Theorem~\ref{falgasbound2} can be improved to $c>2$ then 
we may deduce that there is always an element whose frequency is 
a constant fraction of the number of member-sets. This natural 
weakening of the union-closed sets conjecture is still very much open.
\begin{theorem}[Hu~\cite{HuMaster}]\label{thm:Hu}
Suppose there is a $c>2$ so that any separating union-closed
family $\A'$ with $|\A'|\leq c|U(\A')|$ satisfies the union-closed 
sets conjecture. Then, for every union-closed family $\A$,
there is an element $u$ of frequency 
\[
|\A_u|\geq \frac{c-2}{2(c-1)}|\A|.
\]
\end{theorem}
The theorem is proved along the following lines: 
by cloning some element, the universe $U$
of $\A$ is enlarged to $U'$. At the same 
time, we add sets of the form $U'-x$ in order to 
separate the clones from each other. The resulting family $\A'$ 
is then separating and will be made to have size $|\A'|\leq c|U'|$. 
Now an element of frequency $\geq\tfrac{1}{2}|\A'|$ will still 
have high frequency in $\A$.

\medskip
Falgas-Ravry also gives a family of separating \ucfs~which shows that the 
combination of the
bounds~\eqref{eq:Reimer} and~\eqref{falgasbound} is close to optimal, in the sense that the sum of both bounds can serve as an upper bound on the minimum possible weight of a separable \ucf.
For this, he calls a pair $(m,n)$ \emph{satisfiable} if there is a separating \ucf~with $n$ sets on a universe of $m$ elements.

\begin{theorem}[Falgas-Ravry~\cite{FR11} and Reimer~\cite{Rei03}]\label{upperlowerboundavg}
Let $(m,n)$ be a satisfiable pair of integers.
Let $\A$ be a \ucf~on $m$ elements and $n$ sets of minimal average frequency.
Then
\begin{equation}
\max \left( \frac{n \log_2 n}{2m} , \frac{m+1}{2} \right) \le \frac{1}{m} \cdot \sum_{u\in U(\A)}|\A_u| \le \frac{n \log_2 n}{2m} + \frac{m+1}{2} + \frac{n}{m}.
\end{equation}
\end{theorem}

To establish the upper bound in Theorem~\ref{upperlowerboundavg}, Falgas-Ravry uses  
a construction not unlike that of Duffus and Sands~\cite{DS99}
that we discuss below.

\subsection{Limits of averaging}\label{seclimitsofaveraging}

In the framework of the lattice formulation, 
Cz\'edli, Mar\'oti and Schmidt~\cite{CMS09} construct for every size $m$ of the universe
a family of $\lfloor \tfrac{2}{3}2^{m}\rfloor$ members, for which averaging fails. 
We present here a lattice-free version of their family and a short and elementary proof
that the average is always too small.

On the set $\mathbb N^{<\omega}$ of finite subsets of 
the positive integers, let $\hord$ be the order defined 
by first sorting by increasing largest element and then 
by reverse colex order. 
In other words,  we set
$A\hord B$ if 
\begin{itemize}
\item $\max A < \max B$; or
\item $\max A = \max B$ but $\max (A\Delta B)\in A$
\end{itemize}
for finite $A,B\subseteq\mathbb N$.

As an illustration, here is the initial segment of the order, where we write
$124$ for the set $\{1,2,4\}$:
\begin{align*}
\emptyset < 1 < 12 < 2 < 123 < 23 < 13 < 3 < 1234 < 234 &\\
< 134 < 34 < 124 < 24 < 14 < 4 < 12345 < ...
\end{align*}

For any positive integer $n$, define 
the \emph{Hungarian family $\hfam$}
to be the inital segment of length
$n$
of $\mathbb N^{<\omega}$
under $\hord$. It is easy to check that $\hfam$ is union-closed and that its
universe is $[\lceil \log_2 n \rceil]$.

\begin{theorem}[Cz\'edli, Mar\'oti and Schmidt]\label{hunthm}
For the Hungarian family on $[m]$ of size $n=\lfloor \tfrac{2}{3}2^{m}\rfloor$
\[
\frac{1}{m} \cdot \sum_{i\in [m]}|\hfam_i|<\frac{|\hfam|}{2}.
\]
for any $m>1$.
\end{theorem}
\begin{proof}
The key to the proof are the simple and well-known identities
\begin{align}
\lfloor \tfrac{2}{3}2^m\rfloor &= \frac{2^{m+1}-1}{3} = 2^{m-1}+2^{m-3}+\ldots+4+1\text{ if $m$ odd.}\\
\label{twothirds}
\lfloor \tfrac{2}{3}2^m\rfloor &= \frac{2^{m+1}-2}{3} = 2^{m-1}+2^{m-3}+\ldots+8+2\text{ if $m$ even.}
\end{align}

Put $k=\lfloor\tfrac{m-1}{2}\rfloor$.
Denote by $I_0$ the initial segment of $\mathbb N^{<\omega}$ of length~$2^{m-1}$, 
by $I_1$ the set of the next $2^{m-3}$ sets in the order, 
by $I_2$ the following $2^{m-5}$ sets and so on until we 
reach $I_{k}$.

Clearly, $|I_i|=2^{m-(2i+1)}$ and 
\(
\hfam = I_0\cup I_1\cup \ldots\cup I_{k}.
\)
Moreover, we can see that $I_0=2^{[m-1]}$ and that for $i\geq 1$, the set
$I_i$ is the set of all $X\subseteq [m]$ 
that contain all of $m-1,m-3,\ldots, m-(2i-1)$ and of $m,m-2i$, but
none of $m-2,m-4,\ldots,m-(2i-2)$.

Thus, an element $m-(2i-1)$ appears in half of the members 
of $I_0\cup\ldots\cup I_{i-1}$ and in all of 
the sets in $I_i\cup\ldots\cup I_{k}$.
Its frequency is therefore
\begin{equation}\label{oddfreq}
|\hfam_{m-(2i-1)}|=\tfrac{1}{2}\left(|I_0|+\ldots+|I_{i-1}|\right)
+|I_i|+\ldots+|I_{k}|.
\end{equation}

An element $m-2i$ is contained in half of the sets of $I_0\cup \ldots\cup I_{i-1}$,
in all of the sets in $I_i$ but in none of $I_{i+1}\cup \ldots\cup I_{k}$.
Its frequency is 
\begin{equation}\label{evenfreq}
|\hfam_{m-2i}|=\tfrac{1}{2}\left(|I_0|+\ldots+|I_{i-1}|\right)+|I_i|.
\end{equation}
Moreover, we observe that $m$ lies in all of sets of $\hfam$ but those in $I_0$. 

For the final argument, we assume $m$ to be even, that is $m=2k+2$. The case
of odd $m$ is very similar. With~\eqref{oddfreq} and~\eqref{evenfreq},
we obtain
\begin{align*}
\sum_{j=1}^m|\hfam_j|&=|\hfam_m|+\sum_{i=1}^{k}\left(|\hfam_{m-(2i-1)}|+|\hfam_{m-2i}|\right)+|\hfam_1|\\
&=|\hfam|-|I_0|+\sum_{i=1}^{k}\left(|\hfam|+|I_i|\right)+\frac{1}{2}|\hfam|\\
&=(k+1)|\hfam|-2|I_0|+\frac{3}{2}|\hfam|\\
&=\frac{m}{2}|\hfam|-2^m+\frac{3}{2}\cdot\frac{2^{m+1}-2}{3}
=\frac{m}{2}|\hfam|-1,
\end{align*}
where we used~\eqref{twothirds} in the penultimate step.
\end{proof}

\bigskip

So, the averaging method can never yield the union-closed sets
conjecture in its full generality. Might it perhaps be possible 
to at least obtain the natural relaxation, in which we only 
ask for an element that appears in $\geq 1\%$ of the member-sets?
As Duffus and Sands~\cite{DS99} observed, not even this 
more modest aim may be attained just by averaging.
We present here their construction.

Let $V$ be a set of size $2t$, and $W=\{w_1,\ldots,w_{2^{t}}\}$ be a 
disjoint set of $2^t$ elements. Put
\[
\A=2^V\cup\{V\cup\{w_1,\ldots,w_i\}:i=1\ldots,2^t\}.
\]	
Then $\A$ is a (separating) union-closed family of size $|\A|=2^{2t}+2^t$
on a universe $U=V\cup W$
of size $2t+2^t$. Averaging yields
\begin{align*}
\frac{1}{|U|} \cdot \sum_{u\in U}\frac{|\A_u|}{|\A|}
&=\frac{2t(2^{2t-1}+2^t)+\sum_{i=1}^{2^t}(2^t-i+1)}{(2t+2^t)(2^{2t}+2^t)}\\
&=\frac{2t(2^{2t-1}+2^t)+2^{t-1}(2^t-1)}{(2t+2^t)(2^{2t}+2^t)}\to 0\text{ as }t\to\infty,
\end{align*}
as the largest summand in the numerator is $t2^{2t}$, while the largest
one in the denominator is $2^{3t}$. This shows that an averaging argument cannot
always guarantee an element of frequency at least $c|\A|$ for any $c>0$.

\subsection{Up-compression}\label{secupcompression}

We now outline Reimer's proof of Theorem~\ref{Reimertheorem} because it uses 
a common technique in extremal combinatorics: shifting or compression.
We first restate the theorem.

\newtheorem*{reimerthm}{Theorem~\ref{Reimertheorem}}
\begin{reimerthm}[Reimer~\cite{Rei03}]
Let $\A$ be a \ucf~on $n$ sets. Then
\[
\frac{1}{n} \cdot \sum_{A \in \A} |A| \ge \frac{\log_2 n}{2}.
\]
\end{reimerthm}

Compression subjects the given initial object (the union-closed
family), to small incremental changes until a simpler object is reached
(an up-set), while maintaining the essential properties of the 
initial object. Variants of compression have been used by Frankl in order
to prove the Kruskal-Katona theorem~\cite{Fra84} and in the context of 
traces of finite sets~\cite{Fra83}. 
The technique is also used by Alon~\cite{Alo83} and various others;
see Kalai's blog post~\cite{Kal08} for an enlightening discussion.

Returning to Reimer's proof we define the \emph{up-compression}
of a union-closed family $\mathcal A$. For this, consider an element $i$,
and define 
\[
u_i(A)=\begin{cases}
A+i&\text{ if $A+i\notin\mathcal A$}\\
A&\text{ otherwise},
\end{cases}
\]
for every $A\in\mathcal A$.
Then it turns out that the up-compressed family
 $u_i(\mathcal A):=\{u_i(A):A\in\mathcal A\}$ is still union-closed.
Moreover, iteratively applying up-compression for every element $i$ in the 
universe of $\mathcal A$ results in an \emph{up-set}: a family
$\mathcal U$ on universe $U$ for which $X \in U$ and 
$X \subseteq Y \subseteq U$ implies 
$Y\in\mathcal U$. 
We may always assume $\A$ to have universe $[m]$. 
We then write $u(\mathcal A)$ for the iterated 
up-compression $u_m\circ\ldots\circ u_1(\mathcal A)$. 

\begin{lemma}[Reimer~\cite{Rei03}]
Let $\mathcal A$ be a union-closed family on universe $U$. Then
\begin{enumerate}[\rm (i)]
\item $u_i(\mathcal A)$ is union-closed for any $i\in U$; and
\item $u(\mathcal A)$ is an up-set.
\end{enumerate}
\end{lemma}

\newcommand{\eb}{EB}

What have we gained? The key to the averaging technique is to 
control the total set size $\sum_{A\in\mathcal A} |A|$. For an 
up-set the total set size can be given in a closed form.
Define the \emph{edge boundary} of an up-set $\mathcal U$
on a universe $U$ to 
be 
\[
\eb(\mathcal U)=\{(A,A+i):A\notin\mathcal U,\,i\in U \text{ and } A+i\in\mathcal U\}.
\]
Now
\begin{lemma}[Reimer~\cite{Rei03}]
Let $\mathcal U$ be an up-set on $m$ elements. Then
\[
2\sum_{A\in\mathcal U}|A|=m|\mathcal U|+|\eb(\mathcal U)|.
\]
\end{lemma}

In order to finish Reimer's proof we need to see that the second essential 
part of the compression argument holds: that the object does not 
change too much during compression. Here this means that the 
total set size has  controlled growth.
 
\begin{lemma}[Reimer~\cite{Rei03}]
Let $\mathcal A$ be union-closed family. Then 
\begin{enumerate}[\rm (i)]
\item $\sum_{A\in\mathcal A}|u(A)-A|\leq |\eb(u(\mathcal A))|$; and
\item $\sum_{A\in\mathcal A}|u(A)-A|\leq |\mathcal A|(m-\log_2(|\mathcal A|))$.
\end{enumerate}
\end{lemma}

\begin{proof}[Proof of Theorem~\ref{Reimertheorem}]
Applying the previous lemmas we obtain
\begin{align*}
2\sum_{A\in\mathcal A}|A| =&\, 2\sum_{A\in\mathcal A}|u(A)|-2\sum_{A\in\mathcal A}|u(A)-A|\\
\geq &\, m|u(\mathcal A)|+|\eb(u(\mathcal A))|-2\sum_{A\in\mathcal A}|u(A)-A|\\
\geq &\, m|\mathcal A|+|\eb(u(\mathcal A))|-|\eb(u(\mathcal A))|-|\mathcal A|(m-\log_2(|\mathcal A|))\\
=&\,|\mathcal A|\cdot\log_2(|\mathcal A|).
\end{align*}
\end{proof}

Refining Reimer's approach, Balla, Bollob\'as and Eccles 
improve substantially on Nishimura and Takahashi's observation 
that large union-closed families  never pose a counterexample
to Frankl's conjecture. 

\begin{theorem}[Balla, Bollob\'as and Eccles~\cite{BBE13}]\label{thm:Bollobas}
Any union-closed family on $m$ elements
with at least $\lceil\tfrac{2}{3}2^m\rceil$ member-sets
satisfies the union-closed sets conjecture.
\end{theorem}

In fact, Balla et al.\ prove that the average frequency of such a family $\mathcal A$ is always
at least $\frac{|\mathcal A|}{2}$. In view of Theorem~\ref{hunthm} this is best possible. 

The key idea of the proof of Theorem~\ref{thm:Bollobas} is to exploit the Kruskal-Katona theorem 
in conjunction with up-compression.
This allows to show that, among all union-closed families on $n$ member-sets, 
the Hungarian family $\hfam$ has minimal total set size. 
Since the total set size of $\hfam$ is large, 
provided that $n\geq \lceil\tfrac{2}{3}2^m\rceil$, 
the double-counting argument~\eqref{eq:double-counting} then yields  an
average frequency that is large enough to imply the union-closed
sets conjecture for the given family.

\medskip

Up-compression, and in particular, the effect of the order in which 
the elements $i$ of the universe are chosen  for the up-compression
is further investigated by Rodaro~\cite{Rod12}.
In a fairly involved article with a heavy algebraic flavour he arrives 
at an upper-bound on the number of basis sets of the union-closed family.
(Recall that a non-empty $B\in\A$ is a basis set if $B=A\cup A'$
for $A,A'\in\A$ implies $A=B$ or $A'=B$.) Rodaro's bound, however, 
is weaker than a result of Kleitman from 1976 on set families that 
are union-free. Cast in the language of basis sets of a union-closed family 
the result becomes:

\begin{theorem}[Kleitman~\cite{Kle76}]
Let $\A$ be a union-closed family on $m$ elements. Then the number of basis sets
is at most
\[
{m \choose \lfloor \frac{m}{2} \rfloor} + \frac{2^m}{m}.
\] 
\end{theorem}


While it is not clear how sharp the bound is, 
a family with 
${m \choose \lfloor \frac{m}{2} \rfloor}$ basis sets is easily found: 
simply take all subsets of $2^{[m]}$ of size at least $\lfloor \frac{m}{2} \rfloor$.

\medskip

Up-compression is clearly a powerful concept. So, it seems enticing to apply 
the method in a more direct way to attack Frankl's conjecture: given 
a union-closed family $\A$, choose an element $i$ in its universe and apply
up-compression with respect to $i$, and then reduce the problem to 
the hopefully simpler family $u_i(\A)$. Unfortunately, the up-compressed family $u_i(\A)$
is much too simple with respect to the union-closed sets conjecture: the family satisfies 
it for trivial reasons. Indeed, the element $i$
always appears in at least half of the member-sets of $u_i(\A)$. 

Lo Faro~\cite{LF94b} found a way to circumvent this. Call an element $y$
\emph{dominated by $x$} if $y\in A\in\A$ implies $x\in A$---in other words, 
when $\A_y\subseteq\A_x$. Then we may apply up-compression with respect to $y$
selectively to the sets in $\A_x$. That is, we set 
\[
u'_y(A):=\begin{cases}
A+y&\text{ if $A\in\A_x$ and $A+y\notin \A$}\\
A&\text{ otherwise}.
\end{cases}
\]
The resulting family $\A':=u'_y(\A)$ is still union-closed. Moreover, 
the frequency of $y$ is bounded by the frequency of $x$, which has not 
changed. If $\A'$ satisfies the union-closed sets conjecture 
then this is also the case for the original family $\A$. Thus,
this restricted up-compression allows to force more structure without augmenting
the frequency. While Lo Faro manages to exploit this technique in order 
to obtain a bound on a minimal counterexample it is not clear 
whether it or a variant may be used to a more far-reaching effect.

We note that up-compression is also used by
Leck and Roberts~\cite{LR13} in the context of the union-closed sets conjecture.

\subsection{Generalised averages}

We saw in the previous section that  the Hungarian family $\hfam$ has minimum total set size
among all union-closed families with $n$ member-sets.
Leck, Roberts and Simpson~\cite{LRS12} study a more general set-up, in which they allow 
the set sizes to be weighted. 
For this, they consider non-negative
weight functions $w : 2^{[m]} \to \mathbb{R}_{\ge 0}$ 
that are constant on all sets of the same size. That is, there are reals 
$w_i\geq 0$ so that $w(X)=w_i$ if $|X|=i$, for every $X\subseteq [m]$. 
Moreover, the weights are non-decreasing with $i$, meaning $w_0 \le w_1 \le \ldots \le w_m$.
The \emph{weight} of a non-empty \ucf~$\A$ is then defined as 
$\sum_{A \in \A} w(A)$. 
For example, if $w_i = i$ for all $i \in [0,m]$, then 
$w(\A)$ is just the total set size.

For families generated by $2$-sets, Leck et al.\ managed 
to determine the extremal families. These families turn out to be 
independent of the actual weight. 
In contrast to above, where we used the reverse colex order
we need here the standard colex order:
if $X,Y \subseteq [m]$ are distinct then $X < Y$ if and only if $\max(X \Delta Y) \in Y$.
Then, we define  $\mathcal U_k$ to be the union-closure of the 
first $k$ distinct $2$-sets in the colex order.
For any weight $w$, Leck et al.\ 
calculate the weight of $\mathcal U_k$ to be
\[
\sum_{i=2}^{a+2} \left( {a+1 \choose i} - {a-b \choose i-1} \right) \cdot w_i,
\]
where $a$ and $b$ are any integers such that $0 \le b \le a$ and $k = {a \choose 2} + b$.

\begin{theorem}[Leck, Roberts and Simpson~\cite{LRS12}]\label{thm:Leck}
For every $k$ and every weight~$w$, 
the family $\mathcal U_k$ has minimum weight $w(\mathcal U_k)$
among all \ucfs\ generated by $k$ distinct $2$-sets.
\end{theorem}

A partial result of this had already been proved 
by Imrich, Sauer and Woess~\cite{ISW93}, 
first mentioned in their technical report~\cite{ISW88},
which showed that
any \ucf~$\A$ that is generated by basis sets of size~$2$, 
has an average set size of at least $\tfrac{1}{2}|U(\A)|$.

\medskip
As we observed in Section~\ref{seclimitsofaveraging}, 
averaging does not always succeed, that is, 
the arithmetic mean of the frequencies is sometimes too low to conclude 
that the union-closed sets conjecture holds for a given family. 
For some families, such as the Hungarian family discussed above, this is because 
there is one or perhaps a few elements with very low frequency. 
Those elements might be so rare that, on the whole, the average frequency
drops below the Frankl threshold of half of the  member-sets.

One way to overcome this obstacle is to use a different mean 
than the arithmetic mean, one that de-emphasises the weight of  
extremely rare outliers. This approach has been pursued by Duffus and Sands~\cite{DS99}.
While they consider a quasi-arithmetic mean for the lattice formulation, 
we present here the equivalent form in the set formulation.
In particular, Duffus and Sands pose the question whether there is a $c>1$ so that 
\begin{equation}\label{DSmean}
\frac{1}{|U|}\sum_{u\in U}c^{|\A_u|}\geq c^{\frac{|\A|}{2}}
\end{equation}
for all union-closed families $\A$ 
with universe $U$. Clearly,~\eqref{DSmean} would imply the union-closed 
sets conjecture. As evidence, Duffus and Sands prove that
the lattice version of~\eqref{DSmean}
holds for distributive lattices when $c=4$.

\bigskip
While~\eqref{DSmean} seems quite enticing,  
a new idea is needed to make this, or some other, generalised average work. 
Indeed, it is no longer obvious how the main advantage of the averaging approach
can be exploited, namely that the frequencies are analysed \emph{indirectly} via the 
set sizes. In the case of distributive lattices, Duffus and Sands could 
investigate the individual 
frequencies $|\A_u|$ to arrive at their result. In general, this will not be possible.
For, if it was, then there would be no need to consider a quasiarithmetic mean (or of any 
other kind), as one could immediately exhibit an abundant element.


\subsection{Families of minimum density}

Rather than averaging the frequencies over the whole universe, we may hope to gain more
by restricting the range of the average, for example to the elements 
of the smallest member-set. 
This approach 
was developed by W\'ojcik~\cite{Woj92} and followed up by Balla~\cite{Bal11}.

Define $s_k$ to be the largest real so that for any union-closed family $\A$
and any $k$-set $S$ in $\A$ it holds that 
\begin{equation}\label{eq:WojcikBound}
\frac{1}{|S|}\sum_{u\in S}|\A_u|\geq s_k|\A|.
\end{equation}
The first $10$ values have been determined exactly by W\'ojcik; 
we list here the first five: 
$s_1=\tfrac{1}{2}$, $s_2=\tfrac{1}{2}$, $s_3=\tfrac{4}{9}$, $s_4=\tfrac{2}{5}$
and $s_5=\tfrac{9}{25}$. So, in particular, any $5$-set in any union-closed 
family will always contain an element that appears in at least a third of the 
member-sets.

Somewhat surprisingly, the value $s_k$ coincides with the 
so-called minimal \emph{density} of a family on $k$ elements:
\begin{theorem}[W\'ojcik~\cite{Woj92}]\label{thm:Wojcik}
For every $k\in\mathbb N$ it holds that 
\[
s_k=\min_{\A}\frac{1}{k|\A|} \cdot \sum_{u\in U(\A)}|\A_u|,
\]
where the minimum ranges over all union-closed families $\A$
with $|U(\A)|=k$.
\end{theorem}
We mention that we have reversed here definition and consequence, as 
W\'ojcik defines the $s_k$ as minimal densities but then proves the 
equivalence to~\eqref{eq:WojcikBound}.

W\'ojcik conjectured and Balla proved that:

\begin{theorem}[Balla~\cite{Bal11}]\label{thm:Balla}
For all $k$, $s_k \ge \frac{\log_2 k}{2k}$.
\end{theorem}

The main step in the proof is an application of Reimer's theorem.
As W\'ojcik~\cite{Woj92} indicated, this lower bound is asymptotically optimal.
To see this, consider the family $2^{[r]} \cup [k]$, where $r=\lceil \log_2 k \rceil$,
and observe that its 
density is $(1+o(1)) \frac{\log_2 k}{2k}$.
Note, however, that this family is not separating.

Combining Theorems~\ref{thm:Wojcik} and~\ref{thm:Balla}, 
Balla arrives at a lower bound on the maximum frequency in terms of the size of the universe.

\begin{corollary}[Balla~\cite{Bal11}]
In every \ucf~on $m \ge 16$ elements and $n$ sets there is an 
element contained in at least $\sqrt{\tfrac{\log_2 m}{m}} \cdot \tfrac{n}{2}$ many 
member-sets.
\end{corollary}

\section{Further results}\label{furtherresults}

Sarvate and Renaud~\cite{SR89} observed that if the \ucc~holds for 
\ucfs~on $n$ sets, 
$n$ odd, then it holds for \ucfs~with $n+1$ sets.
In particular, $n_0$ is odd.
Lo Faro~\cite{LF94b} and later Roberts and Simpson~\cite{RSi10} proved $n_0 \ge 4m_0-1$.
As discussed earlier, this result 
turns out to be very useful for families on few sets.\sloppy

Another result in this direction is given by Norton and Sarvate~\cite{NS93}: 
any counterexample with $n_0$ sets contains at least three distinct elements of frequency exactly $\tfrac{n_0-1}{2}$.
Other necessary properties of counterexamples were given by 
Lo Faro~\cite{LF94,LF94b} and Dohmen~\cite{Doh01}.

\medskip
Peng, Sissokho and Zhao~\cite{PSZ12} study 
what they call the \emph{half-life} of set families.
Given a set family $\B$ that is not necessarily union-closed,
they consider the family $\bigcup^k \B$ defined as the family of unions of at most $k$ sets of $\B$.
The half-life of $\B$ is then the least $k$ such that $\bigcup^k \B$
satisfies the assertion of the \ucc.

\section{Extremal frequency}\label{exConway}

Any induction proof of the union-closed sets conjecture will likely 
necessitate a strengthened induction hypothesis coupled with structural
insight on those families with low maximum frequencies. Let us therefore
look at the minimal maximum element frequency a family on a given number
of sets may have.

For a union-closed family $\A$ define $\phi(\A)$ to be the maximum 
frequency of an element of the universe, that is,
\[
\phi(\A)=\max_{u\in U(\A)} |\A_u|.
\] 
Let $\phi(n)$ be the minimum over all $\phi(\A)$, where $\A$ is 
a union-closed family of $n\geq 2$ member-sets. Clearly, this allows the
trivial reformulation of the \ucc\ as:
\begin{conjecture}
$\phi(n)\geq\tfrac{n}{2}$ for all integers $n\geq 2$.
\end{conjecture}

In this way, the union-closed sets conjecture becomes a problem about
an integer sequence.
What can be said about this sequence $\phi(n)$? For instance, that it
is a slowly growing sequence:

\begin{lemma}[Renaud~\cite{Ren91}]
$\phi(n-1)\leq\phi(n)\leq\phi(n-1)+1$ for all $n\geq 2$.
\end{lemma}

Renaud\footnote{
We point out here that our sequence $\phi(n)$ equals 
Renaud's~\cite{Ren91} $\phi(n-1)$.
} used the lemma to compute the first $17$ values of $\phi(n)$.
We put $\phi(1)=1$ so that the sequence starts from $n=1$ on:
\begin{equation}\label{phiseq}
1,1,2,2,3,4,4,4,5,6,7,7,8,8,8,8,9,10,...
\end{equation}

Moreover, if the union-closed sets conjecture is true, then 
$\phi(n)=\tfrac{n}{2}$ if $n$ is a power of two, and 
$\phi(n)>\tfrac{n}{2}$ otherwise, provided 
Poonen's conjecture (Conjecture~\ref{powconj}) is valid as well.

Now, there is a well-known slowly growing integer sequence
that coincides with $\phi(n)$
on  the initial segment~\eqref{phiseq} and that, in addition, 
has $a(n)=\tfrac{n}{2}$ if and only if $n$ is power of two.
This is 
Conway's challenge sequence, defined by
$a(1)=a(2)=1$
and the recurrence relation 
\[
a(n)=a(a(n-1))+a(n-a(n-1)).
\]
See, for instance, Kubo and Vakil~\cite{KB96} for background
on the sequence. 

As Mallows~\cite{Mal91} proved that $a(n)\geq\tfrac{n}{2}$
for all $n\geq 1$, it seems tempting to seek a deeper relation 
between $\phi(n)$ and $a(n)$, and in some sense there is one. 
Renaud and Fitina construct, for every $n$, a union-closed family
whose maximum element frequency is exactly equal to $a(n)$. We discuss
this construction next.

\newcommand{\rfo}{<}
\newcommand{\rfoq}{\leq}
\newcommand{\rf}[1]{\mathcal R(#1)}

Let us define an order $\rfo$ on $\mathbb N^{(<\omega)}$, the set of finite 
subsets of $\mathbb N$, by first sorting by largest element, then 
by decreasing cardinality and finally by colex order.
Thus, 
$A\rfo B$ if
\begin{itemize} \itemsep1pt \parskip0pt 
\item $\max A<\max B$; or
\item $\max A=\max B$ but $|A|>|B|$; or
\item $\max A=\max B$ and $|A|=|B|$
but $\max(A\Delta B)\in B$
\end{itemize}
Omitting parentheses and commas this yields 
\begin{align*}
\emptyset \rfo 1 \rfo 12 \rfo 2 \rfo 123 \rfo 13 \rfo 23 \rfo 3 \rfo 1234 \rfo 124 &\\
\rfo 134 \rfo 234 \rfo 14 \rfo 24 \rfo 34 \rfo 4 \rfo 12345 \rfo ...
\end{align*}
as initial segment. It is easy to see that $A\rfoq C$ and $B\rfoq C$ implies
$A\cup B\rfoq C$, which means that the first $n$ sets of this order form a union-closed
family, denoted by $\rf{n}$. 

\begin{theorem}[Renaud and Fitina~\cite{RF93}]
For every $n\geq 2$,
the most frequent element of the Renaud-Fitina family $\rf{n}$
has frequency $a(n)$, that is, 
\[
\phi(n)\leq a(n).
\]
\end{theorem}

So, is $\phi(n)$ always equal to $a(n)$? By Mallows' result, that would 
clearly prove
the union-closed sets conjecture. Unfortunately, this is not the case. 
In a subsequent paper, Renaud~\cite{Ren95} described families $\mathcal B(n)$ 
whose element frequency is sometimes strictly smaller than Conways' challenge
sequence. This happens for the first time at $n=23$, where $a(n)=14$.
However, no element in the family 
\[
\mathcal B(23)=2^{[4]}\cup\{12345,1235,1245,1345,2345,125,345\}
\]
on $23$ member-sets appears more often than $13$ times. 
We omit the precise construction of $\mathcal B(n)$ but mention that 
it only differs from $\rf{n}$ in the last step, when we delete 
sets of the same size of the power set $2^{[m]}$. There the sets
to delete are chosen in a more balanced way, so that the frequency
of the elements $1,\ldots, m-1$ differs by at most one. 
 
Renaud determines the maximum frequency as follows. Let 
\[
n=2^m-\sum_{i=0}^{r-1}{m-1\choose i}-v,
\]
where $0\leq r < m-1$ and $0\leq v<{m-1\choose r}$. Then
\[
\phi(\mathcal B(n))=2^{m-1}-\sum_{i=0}^{r-2}{m-1\choose i}
-\left\lfloor\frac{rv}{m-1}\right\rfloor
\]
Furthermore, he shows that always $\phi(\mathcal B(n))\leq a(n)$. 
Are the families $\mathcal B(n)$ now truly extremal, that is 
$\phi(n)=\phi(\mathcal B(n))$ for all $n$? Again, this is not 
the case. Renaud gives the example of the family
\[
\mathcal C=2^{[6]}\sm\{6,5,16,25,36,45,136,245\},
\]
in which the most frequent element appears in $30$ member-sets. 
However, in $\mathcal B(56)$ there is an element of frequency~$31$. 

\medskip
To conclude,  we do not know much, in general, about the structure
of an extremal family, nor are there any convincing candidates. 
The only exception are power sets $\mathcal P$, for which holds
$\phi(\mathcal P)=\phi(|\mathcal P|)$,
provided the union-closed sets conjecture is true. 
Nevertheless, the examples in this section seem to indicate
that an extremal family would have relatively few elements compared
to the number of member-sets: let us call a family on 
$n$ member-sets and a universe of size $m$ \emph{compact}
if $2^{m-1}<n\leq 2^m$. For example, 
power sets, the Renaud-Fitina families as well as the Hungarian families
are compact.

\begin{question}
Is it true that for a union-closed family $\A$  
it follows from $\phi(\A)=\phi(|\A|)$ that $\A$ is compact?
\end{question}

An affirmative answer would be a major step towards the union-closed sets
conjecture. Indeed, Reimer's bound~\eqref{Reimeravg} in conjunction with
Theorem~\ref{m=12}  gives:
\begin{observation}
Any compact union-closed family $\A$ 
contains an element that is contained in at least 
$\tfrac{6}{13}|\A|$ member-sets.
\end{observation}

While
we have arrived at the end of this survey, the union-closed sets conjecture
 still has a bit of a journey ahead of it. We hope it will 
be an exciting trip.
\section*{Acknowledgement}

We are grateful for the 
extensive bibliography of Markovi\'c~\cite{Mar07} that was of great help
for our own literature research. 
We thank Bela Bollob\'as,
Dwight Duffus, Peter Frankl, Tomasz \L uczak, Ian Roberts, Jamie Simpson,
Peter Winkler
and David Yost for their input on the history of the conjecture and for 
help in tracking down seemingly lost items of the literature.
We thank Eric Balandraud for inspiring discussions about the Hungarian family.
Finally, we thank the referee who pointed us to the result of Kleitman
in Section~\ref{secupcompression}, and observed that Knill's graph-generated families form lower semimodular lattices.

\small
\bibliographystyle{amsplain}
\bibliography{ucsbib}

\vfill
\small
\vskip2mm plus 1fill
\noindent
Version 25 Oct 2013
\bigbreak

\noindent
Henning Bruhn
{\tt <henning.bruhn@uni-ulm.de>}\\
Universit\"at Ulm, Germany\\[3pt]
Oliver Schaudt
{\tt <schaudto@uni-koeln.de>}\\
Institut f\"ur Informatik\\
Universit\"at zu K\"oln\\
Weyertal 80\\
Germany

\end{document}